\documentclass[11pt]{article}
\usepackage{amsmath, amsthm, amssymb}
\usepackage[margin=3.5cm]{geometry}
\newtheorem{thm}{Theorem}
\newtheorem{prop}[thm]{Proposition}
\newtheorem{lem}[thm]{Lemma}

\newtheorem{conj}[thm]{Conjecture}

\theoremstyle{remark}
\newtheorem{rem}[thm]{Remark}
\newtheorem{notn}[thm]{Notation}

\title{On the Diophantine Equation $2^a3^b + 2^c3^d = 2^e3^f + 2^g3^h$}
\author{Roger Tian}
\date{\today}
\makeatletter
\newcommand\ackname{Acknowledgements}
\if@titlepage
  \newenvironment{acknowledgements}{%
      \titlepage
      \null\vfil
      \@beginparpenalty\@lowpenalty
      \begin{center}%
        \bfseries \ackname
        \@endparpenalty\@M
      \end{center}}%
     {\par\vfil\null\endtitlepage}
\else
  \newenvironment{acknowledgements}{%
      \if@twocolumn
        \section*{\abstractname}%
      \else
        \small
        \begin{center}%
          {\bfseries \ackname\vspace{-.5em}\vspace{\z@}}%
        \end{center}%
        \quotation
      \fi}
      {\if@twocolumn\else\endquotation\fi}
\fi
\makeatother
\begin{document}
\maketitle
\begin{abstract}
This paper is a continuation of \cite{friedprop}, in which I studied Harvey Friedman's problem of whether the function $f(x,y) = x^2 + y^3$ satisfies any identities; however, no knowledge of \cite{friedprop} is necessary to understand this paper. We will break the exponential Diophantine equation $2^a3^b + 2^c3^d = 2^e3^f + 2^g3^h$ into subcases that are easier to analyze. Then we will solve an equation obtained by imposing a restriction on one of these subcases, after which we will solve a generalization of this equation.
\end{abstract}
\begin{acknowledgements}
I would like to thank my thesis advisor, George Bergman, for giving advice on how to better organize this paper and make it more readable, and for pointing out areas of my paper that needed clarification.
\end{acknowledgements}

We will follow the convention that $0 \notin \mathbb{N}$.

A resticted version of Friedman's problem (mentioned in the Abstract) I studied in my paper \cite{friedprop} is related to the solution set of the equation \begin{equation}\label{deg equation three} 2^a3^b + 2^c3^d = 2^e3^f + 2^g3^h. \end{equation} The results we will get on this equation, which are still very partial, will not be applied in this paper to Friedman's problem. 

Suppose $a_0$, $b_0$, $c_0$, $d_0$, $e_0$, $f_0$, $g_0$, $h_0$ are nonnegative integers such that \begin{equation}\label{deg equation four} 2^{a_0}3^{b_0} + 2^{c_0}3^{d_0} = 2^{e_0}3^{f_0} + 2^{g_0}3^{h_0}. \end{equation} Without loss of generality, we may assume that $\min\{a_0,c_0,e_0,g_0\} = 0 = \min\{b_0,d_0,f_0,h_0\}$, or equivalently $0 \in \{a_0,c_0,e_0,g_0\}$ and $0 \in \{b_0,d_0,f_0,h_0\}$; we can always divide (\ref{deg equation four}) by $2^{\min\{a_0,c_0,e_0,g_0\}}3^{\min\{b_0,d_0,f_0,h_0\}}$. Suppose that there is exactly one zero in $\{a_0,c_0,e_0,g_0\}$. Then (\ref{deg equation four}) reduces to an equation in which exactly three of its four terms contain a factor of 2, so one side of this resulting equation is divisible by 2 while the other side is not, which is a contradiction. Thus, there must be at least two zeros in $\{a_0,c_0,e_0,g_0\}$ and, by the same reasoning, with the factor 2 replaced by the factor 3, there must be at least two zeros in $\{b_0,d_0,f_0,h_0\}$. Then, depending on which terms of (\ref{deg equation four}) the zeros occur in, we can reduce Equation (\ref{deg equation three}) to 36 cases. However, merging the cases that are identical up to permutations of the summands, we get the following seven equations: 
\begin{equation}\label{degeqn five} 1 + 1 = 2^e3^f + 2^g3^h \end{equation}
\begin{equation}\label{degeqn six} 1 + 3^d = 2^e + 2^g3^h \end{equation}
\begin{equation}\label{degeqn seven} 3^b + 3^d  = 2^e + 2^g \end{equation}
\begin{equation}\label{degeqn eight} 1 + 2^c = 3^f + 2^g3^h \end{equation}
\begin{equation}\label{degeqn nine} 1 + 2^c3^d = 1 + 2^g3^h \end{equation}
\begin{equation}\label{degeqn ten} 3^b + 2^c = 1 + 2^g3^h \end{equation}
\begin{equation}\label{degeqn eleven} 3^b + 2^c = 3^f + 2^g \end{equation}
Note, for instance, that in the case $a_0 = b_0 = c_0 = f_0 = 0$, Equation (\ref{degeqn six}) must have at least one solution. The solution to (\ref{degeqn five}) is $(e,f,g,h) = (0,0,0,0)$. The solutions to (\ref{degeqn nine}) are $(c,d,g,h) = (s,t,s,t)$ for all nonnegative integers $s$ and $t$. We now solve (\ref{degeqn ten}) subject to the restiction $b = h$, i.e. the equation $2^c - 1 = 3^b(2^g - 1)$. We first prove a few lemmas.

\begin{lem}
\label{cyclotomic division}
Let $p, m, n \in \mathbb{N} \cup \{0\}$ where $p > 1$ and $m > 0$. If $p^m - 1 \mid p^n - 1$, then $m \mid n$.
\end{lem}
\begin{proof}
We have $n = qm + r$ where $q, r \in \mathbb{N} \cup \{0\}$ and $0 \le r < m$. We will prove this lemma by induction on $q$. For $q = 0$, we have $n < m$, so $p^n - 1 < p^m - 1$, from which it follows that $p^m - 1 \mid p^n - 1$ $\Longrightarrow$ $p^n - 1 = 0$ $\Longrightarrow$ $n = 0$ $\Longrightarrow$ $m \mid n$. Suppose the lemma is true for some $q$, we will prove it for $q + 1$. Suppose $p^m - 1 \mid p^n - 1 = p^{(q+1)m+r} - 1$. Then $p^m - 1$ divides $p^{(q+1)m+r} - 1 - (p^m - 1) = p^{(q+1)m+r} - p^m = p^m(p^{qm+r} - 1)$. Since $p^m - 1$ and $p^m$ are relatively prime, we have $p^m - 1 \mid p^{qm+r} - 1$. It follows from our inductive hypothesis that $m \mid qm + r$, so $r = 0$ and $n = (q+1)m$. This completes the induction.
\end{proof}

\begin{notn}
\label{div max}
Let $n, m, k \in \mathbb{N}.$ By $n^k \parallel m$ we will always mean that $n^k \mid m$ and $n^{k+1} \nmid m$.
\end{notn}

\begin{lem}
\label{power three minus}
If $k, n \in \mathbb{N}$ where $k$ is odd, then $2^{n+2} \parallel 3^{2^nk} - 1$.
\end{lem}
\begin{proof}
We will first prove this claim for $n = 1$. We have $3^{2k} - 1 = 9^k - 1 = (8 + 1)^k - 1 = -1 + \sum_{j=0}^k{{k \choose j}8^j} = \sum_{j=1}^k{{k \choose j}8^j} = 8 \sum_{j=1}^k{{k \choose j}8^{j-1}}$, and we see that $\sum_{j=1}^k{{k \choose j}8^{j-1}}$ is odd because the $j = 1$ term of this sum is odd and all other terms of this sum are even. Now suppose the claim is true for some $n \ge 1$. Then there exists an $l \in \mathbb{N}$ such that $l$ is odd and $3^{2^{n+1}k} - 1 = (3^{2^nk})^2 - 1 = (3^{2^nk} - 1 + 1)^2 - 1 = (2^{n+2}l + 1)^2 - 1 = 2^{2(n+2)}l^2 + 2(2^{n+2}l) + 1 - 1 = 2^{2(n+2)}l^2 + 2(2^{n+2}l) = 2(2^{n+2}l)(2^{n+1}l + 1)$, and we see that $2^{n+1}l + 1$ is odd and $2^{n+3} \parallel 3^{2^{n+1}k} - 1$. This completes the induction.
\end{proof}

\begin{lem}
\label{power three plus}
If $m$ is odd, then $2^2 \parallel 3^m + 1$.
\end{lem}
\begin{proof}
Notice that for $m = 1$ we have $3^m + 1 = 3 + 1 = 4$. Now suppose that the claim is true for some $m \ge 1$, where $m$ is odd. We have $3^{m+2} + 1 = 9 \cdot 3^m + 1 = 9(3^m + 1 - 1) + 1 = 9(4k - 1) + 1 = 36k - 9 + 1 = 36k - 8 = 4(9k - 2)$ where $k$ is odd, and we see that $9k - 2$ is also odd. This completes the induction.
\end{proof}

\begin{lem}
\label{power two minus}
If $m_1, l \in \mathbb{N}$ and $m_2 \in \mathbb{N} \cup \{0\}$ where $2, 3 \nmid l$, then $3^{m_2+1} \parallel 2^{2^{m_1}3^{m_2}l} - 1$.
\end{lem}
\begin{proof}
We will first prove this claim for $m_2 = 0$. Notice that $2^{2^{m_1}l} - 1 = 4^{2^{m_1-1}l} - 1 = (3 + 1)^{2^{m_1-1}l} - 1 = \sum_{i=1}^{2^{m_1-1}l}{{2^{m_1-1}l \choose i}3^i} + 1 - 1 = \sum_{i=1}^{2^{m_1-1}l}{{2^{m_1-1}l \choose i}3^i} = \sum_{i=2}^{2^{m_1-1}l}{{2^{m_1-1}l \choose i}3^i} + 2^{m_1-1}l \cdot 3 = 3(\frac{1}{3} \sum_{i=2}^{2^{m_1-1}l}{{2^{m_1-1}l \choose i}3^i} + 2^{m_1-1}l)$, and we see that $3 \nmid \frac{1}{3} \sum_{i=2}^{2^{m_1-1}l}{{2^{m_1-1}l \choose i}3^i} + 2^{m_1-1}l$. Now suppose for some $m_2 \ge 0$ we have $2^{2^{m_1}3^{m_2}l} - 1 = 3^{m_2+1}l_1$ where $3 \nmid l_1$. Then $2^{2^{m_1}3^{m_2}l} = 3^{m_2+1}l_1 + 1 \Longrightarrow 2^{2^{m_1}3^{m_2+1}l} = (3^{m_2+1}l_1 + 1)^3 = 3^{3(m_2+1)}l_1^3 + 3 \cdot 3^{2(m_2+1)}l_1^2 + 3 \cdot 3^{m_2+1}l_1 + 1$, so $2^{2^{m_1}3^{m_2+1}l} - 1 = 3 \cdot 3^{m_2+1}l_1(3^{2(m_2+1)-1}l_1^2 + 3^{m_2+1}l_1 + 1)$. Since $3 \nmid l_1(3^{2(m_2+1)-1}l_1^2 + 3^{m_2+1}l_1 + 1)$, we have $3^{m_2+2} \parallel 2^{2^{m_1}3^{m_2+1}l} - 1$. This completes the induction.
\end{proof}

\begin{lem}
\label{power two plus}
If $m_1, l \in \mathbb{N} \cup \{0\}$ where $2, 3 \nmid l$, then $3^{m_1+1} \parallel 2^{3^{m_1}l} + 1$.
\end{lem}
\begin{proof}
We will first prove this claim for $m_1 = 0$. Notice that $2^l + 1 = (3 - 1)^l + 1 = \sum_{k=1}^l{{l \choose k}3^k(-1)^{l-k}} + (-1)^l + 1 = \sum_{k=1}^l{{l \choose k}3^k(-1)^{l-k}} - 1 + 1 = \sum_{k=1}^l{{l \choose k}3^k(-1)^{l-k}} = 3 \sum_{k=1}^l{{l \choose k}3^{k-1}(-1)^{l-k}} = 3(\sum_{k=2}^l{{l \choose k}3^{k-1}(-1)^{l-k}} + l)$, and we see that $3 \nmid \sum_{k=2}^l{{l \choose k}3^{k-1}(-1)^{l-k}} + l$. Suppose for some $m_1 \ge 0$ we have $2^{3^{m_1}l} + 1 = 3^{m_1 +1}l_1$ where $3 \nmid l_1$. Then we have $2^{3^{m_1}l} = 3^{m_1 + 1}l_1 - 1 \Longrightarrow 2^{3^{m_1+1}l} = (3^{m_1+1}l_1 - 1)^3 = 3^{3(m_1+1)}l_1^3 - 3 \cdot 3^{2(m_1+1)}l_1^2 + 3 \cdot 3^{m_1+1}l_1 - 1$, so $2^{3^{m_1+1}l} + 1 = 3 \cdot 3^{m_1+1}l_1(3^{2(m_1+1)-1}l_1^2 - 3^{m_1+1}l_1 + 1)$. Since $3 \nmid l_1(3^{2(m_1+1)-1}l_1^2 - 3^{m_1+1}l_1 + 1)$, we have $3^{m_1+2} \parallel 2^{3^{m_1+1}l} + 1$. This completes the induction.
\end{proof}

\begin{prop}
\label{cyclotomic diophantine}
The only solutions $(k,m,n)$ in the positive integers of the exponential Diophantine equation $3^k(2^m - 1) = 2^n - 1$ are $(1,1,2)$ and $(2,3,6)$.
\end{prop}
\begin{proof}
We know from Lemma \ref{cyclotomic division} that $n = lm$ for some $l \in \mathbb{N}$. Since $k > 0$, we must have $l \ge 2$. Notice that $2^n - 1 = 2^{lm} - 1 = (2^m - 1)(2^{(l-1)m} + 2^{(l-2)m} + \ldots + 2^m + 1) = 3^k(2^m - 1)$, so \begin{equation}\label{newlabel14} 3^k = 2^{(l-1)m} + 2^{(l-2)m} + \ldots + 2^m + 1 > 2^m - 1. \end{equation} It follows that $3^{2k} > 3^k(2^m - 1)$, so \begin{equation}\label{inequality one} 3^{2k} > 2^n - 1. \end{equation}

Now, by Lemma \ref{power two minus} we have $3^k \mid 2^n - 1$ $\Longrightarrow$ $n = 2^{m_1}3^{m_2}l_1$ where $m_1, l_1 \in \mathbb{N}$ and $m_2 \in \mathbb{N} \cup \{0\}$ such that $2, 3 \nmid l_1$ and $m_2 \ge k - 1$. Note that $m_2 = k - 1$ if $3^k \parallel 2^n - 1$. 

We shall now prove by induction that $3^{2k} < 2^{2^{m_1}3^{k-1}l_1} - 1$ for every $k \ge 3$ and all choices of $m_1$ and $l_1$. It is easy to check that, for all choices of $m_1$ and $l_1$, we have $3^{2k} < 2^{2^{m_1}3^{k-1}l_1} - 1$ for $k = 3$. Suppose we know, for some value of $k \ge 3$, that $3^{2k} < 2^{2^{m_1}3^{k-1}l_1} - 1$, or equivalently $3^{2k} + 1 < 2^{2^{m_1}3^{k-1}l_1}$, for all choices of $m_1$ and $l_1$. Then we have $2^{2^{m_1}3^kl_1} = (2^{2^{m_1}3^{k-1}l_1})^3 > (3^{2k} + 1)^3 > 3^{2(k+1)} + 1$ for all choices of $m_1$ and $l_1$. This completes the induction. If $k \ge 3$ and $3^k \mid 2^n - 1$, then $2^n - 1 = 2^{2^{m_1}3^{m_2}l_1} - 1 \ge 2^{2^{m_1}3^{k-1}l_1} - 1 > 3^{2k}$, contradicting (\ref{inequality one}). Thus, we see that there are no solutions for $k \ge 3$. 

It remains to determine the possible solutions when $k = 1, 2$. Suppose $k = 1$. Then $3^{2k} = 3^2 > 2^n - 1$ implies that $n = 1$, 2, or 3, so we must have $n = 2$ because $n$ is even, hence $m = 1$. Suppose $k = 2$. Then by (\ref{newlabel14}) we have $3^2 = 9 > 2^m - 1$, so $m = 1$, 2, or 3. It is easy to check that the cases $m = 1$ and $m = 2$ do not yield solutions. For $m = 3$, we have $n = 6$.
\end{proof}
\begin{rem}
\label{diophantine explanation}
Central to our proof is the fact that, for $k$ sufficiently large, $2^n - 1 = 3^k \cdot q$ implies $q$ must be much larger than $3^k$.
\end{rem}

We will now solve a more general exponential Diophantine equation using the same ideas in the proof of the preceding proposition. First, we generalize Lemma \ref{power two minus}.
\begin{lem}
\label{power odd minus}
Let $m \ge 3$ be an odd positive integer. Then the following statements hold.
\begin{enumerate}
	\item If $n \in \mathbb{N}$ is odd, then $m \nmid (m - 1)^n - 1$.
	\item If $m_1, l \in \mathbb{N}$, $m_2 \in \mathbb{N} \cup \{0\}$, and $2, m \nmid l$, then $m^{m_2+1} \parallel (m - 1)^{2^{m_1}m^{m_2}l} - 1$.
\end{enumerate}
\end{lem}
\begin{proof}
We consider the two cases separately.
\begin{enumerate}
	\item We have $(m - 1)^n - 1 \equiv (-1)^n - 1 \equiv -1 - 1 \equiv -2 \pmod{m}$, from which the conclusion follows.
	\item Our proof is by induction on $m_2$. 
	
	For $m_2 = 0$, we have $(m - 1)^{2^{m_1}l} - 1 = \sum_{i=0}^{2^{m_1}l}{{2^{m_1}l \choose i}m^i(-1)^{2^{m_1}l-i}} - 1 = -1 + 1 + \sum_{i=1}^{2^{m_1}l}{{2^{m_1}l \choose i}m^i(-1)^{2^{m_1}l-i}} = \sum_{i=1}^{2^{m_1}l}{{2^{m_1}l \choose i}m^i(-1)^{2^{m_1}l-i}} =$ \begin{equation}\label{m div max} m\left(\left(\sum_{i=2}^{2^{m_1}l}{{2^{m_1}l \choose i}m^{i-1}(-1)^{2^{m_1}l-i}}\right) - 2^{m_1}l\right). \end{equation} Notice that $m \mid \sum_{i=2}^{2^{m_1}l}{{2^{m_1}l \choose i}m^{i-1}(-1)^{2^{m_1}l-i}}$, but $m \nmid 2^{m_1}l$ because $m \nmid l$ and $\gcd(m,2^{m_1}) = 1$. Thus, $m \nmid \sum_{i=2}^{2^{m_1}l}{{2^{m_1}l \choose i}m^{i-1}(-1)^{2^{m_1}l-i}} - 2^{m_1}l$ and so $m \parallel (m - 1)^{2^{m_1}l} - 1$ as desired. 
	
	Suppose for some $m_2 \ge 0$ we have $(m - 1)^{2^{m_1}m^{m_2}l} - 1 = m^{m_2+1}l_1$ where $m \nmid l_1$. Then $(m - 1)^{2^{m_1}m^{m_2}l} = m^{m_2+1}l_1 + 1$ $\Longrightarrow$ $(m - 1)^{2^{m_1}m^{m_2+1}l} = (m^{m_2+1}l_1 + 1)^m = \sum_{i=0}^m{{m \choose i}(m^{m_2+1}l_1)^i} = 1 + \sum_{i=1}^m{{m \choose i}(m^{m_2+1}l_1)^i}$, so $(m - 1)^{2^{m_1}m^{m_2+1}l} - 1 = \sum_{i=1}^m{{m \choose i}(m^{m_2+1}l_1)^i} = m \cdot m^{m_2+1}l_1 + \sum_{i=2}^m{{m \choose i}(m^{m_2+1}l_1)^i} = m^{m_2+2}l_1 + \sum_{i=2}^m{{m \choose i}m^{(m_2+1)i}l_1^i} = m^{m_2+2}(l_1 + \sum_{i=2}^m{{m \choose i}m^{(m_2+1)(i-1)-1}l_1^i})$. By assumption $m \nmid l_1$. However, $m \mid \sum_{i=2}^m{{m \choose i}m^{(m_2+1)(i-1)-1}l_1^i}$ because $m$ clearly divides ${m \choose i}m^{(m_2+1)(i-1)-1}l_1^i$ for all $i \ge 3$ and ${m \choose 2}m^{m_2}l_1^2 = \frac{m(m - 1)}{2}m^{m_2}l_1^2 = \frac{m-1}{2}m^{m_2+1}l_1^2$ is divisible by $m$ since $m - 1$ is even. It follows that $m \nmid l_1 + \sum_{i=2}^m{{m \choose i}m^{(m_2+1)(i-1)-1}l_1^i}$. This completes the induction.
\end{enumerate}
\end{proof}
\begin{rem}
\label{no power even}
Part 1 of the lemma holds for all integers $m \ge 3$. If $m$ is even, then we can write $m = 2^kp$ where $k, p \in \mathbb{N}$ and $2 \nmid p$, and we can observe (taking $m_1 = k$ and $l = p$) from (\ref{m div max}) that $m^2 \mid (m - 1)^{2^kp} - 1$. Thus, Part 2 of the lemma is false for even $m$.
\end{rem}

\begin{prop}
\label{cyclotomic diophantine gen}
The only solutions $(k,p,q,n)$ in the positive integers of the exponential Diophantine equation $(2n + 1)^k((2n)^p - 1) = (2n)^q - 1$ are $(2,3,6,1)$ and $(1,1,2,n)$ for all positive $n$.
\end{prop}
\begin{proof}
We know from Lemma \ref{cyclotomic division} that $q = lp$ for some $l \in \mathbb{N}$. Since $k > 0$, we must have $l \ge 2$. Notice that $(2n)^q - 1 = (2n)^{lp} - 1 = ((2n)^p - 1)((2n)^{(l-1)p} + (2n)^{(l-2)p} + \ldots + (2n)^p + 1) = (2n + 1)^k((2n)^p - 1)$, so we have $(2n + 1)^k = (2n)^{(l-1)p} + (2n)^{(l-2)p} + \ldots + (2n)^p + 1 > (2n)^p - 1$. It follows from \begin{equation}\label{added inequality} (2n + 1)^k > (2n)^p - 1 \end{equation} that $(2n + 1)^{2k} > (2n + 1)^k((2n)^p - 1)$, so \begin{equation}\label{inequality two} (2n + 1)^{2k} > (2n)^q - 1. \end{equation} 

Since $(2n + 1)^k \mid (2n)^q - 1$, by Lemma \ref{power odd minus} we have $q = 2^{m_1}(2n + 1)^{m_2}l_1$ where $m_1, l_1 \in \mathbb{N}$ such that $2, 2n + 1 \nmid l_1$ and $m_2 \in \mathbb{N} \cup \{0\}$ such that $m_2 \ge k - 1$. 

We will prove by induction that for every $k \ge 3$ we have $(2n + 1)^{2k} < (2n)^{2^{m_1}(2n+1)^{k-1}l_1} - 1$ for all choices of $n$, $m_1$, and $l_1$. For $k = 3$, observe that $(2n + 1)^{2k} < (2n)^{2^{m_1}(2n+1)^{k-1}l_1} - 1$ for all choices of $n$, $m_1$, and $l_1$; take $x := 2n + 1$, $m_1 = 1 = l_1$ and note that $x^6 < x^{x^2} - 1 < ((x - 1)^2)^{x^2} - 1 = (x - 1)^{2x^2} - 1$ for all $x \ge 3$. Suppose we know, for some value of $k \ge 3$, that $(2n + 1)^{2k} < (2n)^{2^{m_1}(2n+1)^{k-1}l_1} - 1$ $\Longleftrightarrow$ $(2n + 1)^{2k} + 1 < (2n)^{2^{m_1}(2n+1)^{k-1}l_1}$ for all choices of $n$, $m_1$, and $l_1$. Then $(2n)^{2^{m_1}(2n+1)^kl_1} > ((2n + 1)^{2k} + 1)^{2n+1} \ge (2n + 1)^{2k(2n+1)} + 1 \ge (2n + 1)^{2(k+1)} + 1$ for all choices of $n$, $m_1$, and $l_1$. This completes the induction. If $k \ge 3$ and $(2n + 1)^k \mid (2n)^q - 1$, then we have $(2n)^q - 1 = (2n)^{2^{m_1}(2n + 1)^{m_2}l_1} - 1 \ge (2n)^{2^{m_1}(2n+1)^{k-1}l_1} - 1 > (2n + 1)^{2k}$, contradicting (\ref{inequality two}). Thus, there are no solutions for $k \ge 3$. 

It remains to determine the possible solutions when $k = 1, 2$. Consider $k = 1$. Then by (\ref{added inequality}) we have \begin{equation}\label{newlabel6} 2n + 1 > (2n)^p - 1, \end{equation} which for $n = 1$ becomes $4 > 2^p$, which is false for every $p \ge 2$. Notice that as $n$ increases, $(2n)^p - 1$ where $p \ge 2$ increases faster than $2n + 1$. It follows that (\ref{newlabel6}) is false for every $p \ge 2$, so we must have $p = 1$. For $p = 1$, we have $(2n + 1)^k((2n)^p - 1) = (2n + 1)(2n - 1) = (2n)^2 - 1 = (2n)^q - 1$ $\Longrightarrow$ $q = 2$, and $n$ can be any positive integer. Consider $k = 2$. Then by (\ref{added inequality}) we have \begin{equation}\label{newlabel7} (2n + 1)^2 > (2n)^p - 1, \end{equation} which for $n = 1$ becomes $9 > 2^p - 1$, which is false for every $p \ge 4$. Notice that as $n$ increases, $(2n)^p - 1$ where $p \ge 4$ increases faster than $(2n + 1)^2$. It follows that (\ref{newlabel7}) is false for every $p \ge 4$, so we must have $p = 1$, $2$, or $3$. If $p = 3$, then by (\ref{newlabel7}) we must have $n = 1$, so $3^2(2^3 - 1) = 2^q - 1$ $\Longrightarrow$ $q = 6$. Take $x:= 2n$ for simplicity of notation as we check the remaining two cases. Suppose $p = 2$. Then we have $(x + 1)^2(x^2 - 1) = x^q - 1$ $\Longrightarrow$ $x^4 + 2x^3 -2x - 1 = x^q - 1$ $\Longrightarrow$ $x^4 + 2x^3 - 2x = x^q$, so $q \ge 5$. However, it is easy to see that $x^4 + 2x^3 - 2x < x^q$ for $q \ge 5$, so there are no solutions for $p = 2$. Suppose $p = 1$. Then $(x + 1)^2(x - 1) = x^q - 1$ $\Longrightarrow$ $x^3 + x^2 - x - 1 = x^q - 1$ $\Longrightarrow$ $x^3 + x^2 - x = x^q$, so $q \ge 4$. However, it is easy to see that $x^3 + x^2 - x < x^q$ for $q \ge 4$, so there are no solutions for $p = 1$.
\end{proof}

\begin{notn}
\label{newlabel9}
For all $n, m \in \mathbb{N}$ where $n \ge m \ge 2$ there exists a unique $k \in \mathbb{N} \cup \{0\}$ such that $m^k \parallel n$, and we will denote this $k$ by $v_m(n)$.
\end{notn}

Since Lemma \ref{power odd minus} allowed us to solve the Diophantine equation of Proposition \ref{cyclotomic diophantine gen}, a natural question is whether this lemma can be extended to \textit{all} positive integers $m \ge 3$. If this lemma can be extended thus, then we may be able to solve the more general Diophantine equation \begin{equation}\label{newlabel12} (m + 1)^k(m^p - 1) = m^q - 1 \end{equation} using arguments similar to those in the proof of Proposition \ref{cyclotomic diophantine gen}. There may be many different such extensions, some of them stronger than others. For the time being, we state one such possible (rather weak) extension for even integers $m > 3$ as the following

\begin{conj}
\label{newlabel10}
There exists a positive integer $N$ such that for all integers $n \ge N$ and all even integers $m > 3$ we have $(m^{v_m((m-1)^n - 1)})^2 \le (m-1)^n - 1$.
\end{conj}

\end{document}